\documentclass[10pt,a4paper]{article}
\usepackage[margin=1.25cm]{geometry}
\usepackage{amsmath,amsthm,amsfonts,amssymb,amscd,cite,graphicx,epstopdf}
\usepackage{amssymb}
\usepackage{stmaryrd}
\usepackage{float}

\usepackage{titlesec}
\titleformat{\section}
{\normalfont\fontsize{12}{15}\bfseries}{\thesection}{1em.}{}

\newtheorem{proposition}{Proposition}[section]

\newtheorem{corollary}{Corollary}[section]

\newtheorem{definition}{Definition}[section]
\newtheorem{theorem}{Theorem}[section]
\newtheorem{remark}{Remark}[section]
\newtheorem{example}{Example}[section]

\allowdisplaybreaks[4]

\let\oldbibliography\thebibliography
\renewcommand{\thebibliography}[1]{%
	\oldbibliography{#1}%
	\setlength{\itemsep}{-2pt}%
}

\begin{document}
	
	\baselineskip=0.20in


	\noindent
	{\large \bf Eigenvalues of Hypergraph Products and Reciprocal Eigenvalue Property}\\

	\noindent
	Shashwath S Shetty$^{1}$, K Arathi Bhat$^{2,}\footnote{Corresponding author (arathi.bhat@manipal.edu)}$\\
	
	\noindent
	\footnotesize $^{1,2}${\it Manipal Institute of Technology, Manipal Academy of Higher Education, Manipal, India}\\


	\setcounter{page}{1} \thispagestyle{empty}

	\baselineskip=0.20in

	\normalsize
	
	\begin{abstract}
		\noindent
		
		Spectral hypergraph theory has recently attracted considerable interest as it provides a natural framework for modeling higher-order relationships beyond classical graphs. In this setting, eigenvalues of adjacency, Laplacian, and signless-Laplacian hypermatrices play an important role in understanding the underlying structure of hypergraphs. In this work, we study the adjacency, Laplacian, and signless-Laplacian eigenvalues of the join, Kronecker, and corona products of hypergraphs, and examine how these spectra behave under such operations. These investigations help in better understanding the interplay between hypergraph structure and spectral properties. The reciprocal eigenvalue property is of particular interest due to the spectral symmetries it reflects. Motivated by this, we extend the notion of reciprocal eigenvalue property to hypergraphs and show that power hypertrees do not satisfy this property.
		\\[2mm]
		{\bf Keywords:} join; corona; Kronecker; main eigenvalue; reciprocal eigenvalue.\\[2mm]
		{\bf 2020 Mathematics Subject Classification:} 05C65, 15A69, 15A18.
	\end{abstract}
	
	\baselineskip=0.20in

	\section{Introduction}
	The study of spectra of graph products is a well-established topic in algebraic graph theory, providing explicit relationships between the eigenvalues of product graphs and those of their factors. In graph theory, various graph products are employed to construct the graph with desired property.
	For graphs $G_1$ and $G_2$ with adjacency matrices $A(G_1)$ and $A(G_2)$, many standard graph products admit representations in terms of Kronecker-type operations. For instance, the adjacency matrix of the Cartesian product satisfies $A(G_1 \square G_2) = A(G_1)\otimes I + I \otimes A(G_2)$, which implies that its eigenvalues are of the form $\lambda_i + \mu_j$, where $\lambda_i \in \sigma(A(G_1))$ and $\mu_j \in \sigma(A(G_2))$. Similarly, for the tensor (Kronecker) product, one has $A(G_1 \times G_2) = A(G_1)\otimes A(G_2)$, yielding eigenvalues $\lambda_i \mu_j$.
	For the detailed survey on the spectrum of various graph products, one can refer to \cite{barik2018spectra}.
	These structured relationships illustrate how spectral properties are preserved and combined under graph products, and they serve as a foundation for extending spectral analysis to uniform hypergraphs.

	Eigenvalues of hypergraph products are studied by extending spectral theory from matrices to higher-order tensors associated with hypergraphs. Given two uniform hypergraphs $H_1$ and $H_2$, the adjacency (Laplacian, signless-Laplacian) hypermatrix of their product hypergraph (some) inherits structured patterns from the corresponding hypermatrices of $H_1$ and $H_2$. The spectrum of the product hypergraph can then be related to the spectra of $H_1$ and $H_2$ through algebraic identities that generalize classical results for graph products.
	By the eigenvalue of the hypergraph, we mean the eigenvalue of the adjacency hypermatrix of the hypergraph. 
	If $\lambda_1$ is an eigenvalue of $\mathcal{H}_1$ and $\lambda_2$ is an eigenvalue of $\mathcal{H}_2$, it was shown that
	
	\begin{itemize}
		\item \cite{cooper2012spectra} $\lambda_1, \lambda_2$ are the eigenvalues of the vertex-disjoint union of $\mathcal{H}_1$ and $\mathcal{H}_2,$ 
		\item \cite{cooper2012spectra} $\lambda_1 + \lambda_2$ is an eigenvalue of the Cartesian product of $\mathcal{H}_1$ and $\mathcal{H}_2,$ and 
		\item \cite{pearson2013eigenvalues,shao2013general} $\lambda_1\lambda_2$ is an eigenvalue of the Kronecker product of $\mathcal{H}_1$ and $\mathcal{H}_2. $
	\end{itemize}
	
	In this work, we consider the problem of determining the relation between the adjacency, Laplacian, and signless-Laplacian eigenvalues of the join, Kronecker, and corona products of two hypergraphs with the eigenvalues of the factored hypergraphs.

	A graph is said to satisfy the \textit{reciprocal eigenvalue property} ($R$-property) \cite{barik2007spectrum} if $\lambda$ is in the spectrum of the graph implies $\frac{1}{\lambda}$ is also in the spectrum of the graph. 
	A graph is said to satisfy the \textit{strong-reciprocal eigenvalue property} ($SR$-property) \cite{barik2007spectrum} if the existence of $\lambda$ as an eigenvalue of the graph with multiplicity $m_{\lambda}$ implies the existence of the eigenvalue $\frac{1}{\lambda}$ with the same multiplicity $m_{\lambda}$. It is easy to see that the graphs satisfying both the $R$-property and $SR$-property are non-singular, as the graphs under consideration are finite. 
	Also, a singular graph is said to satisfy the \textit{weak-reciprocal eigenvalue property} \cite{barik2007spectrum} if the non-zero part of the spectrum satisfies the $R$-property.

	In the case of trees, Barik and Pati \cite{barik2006nonsingular} have proved that $R$-property and $SR$-property are equivalent, and a complete characterization of trees satisfying this property (corona trees) was obtained. The characterization of unicyclic graphs with $SR$-property was carried out in \cite{barik2008unicyclic}.
	The weighted trees (of order at least 8) satisfying $R$-property were characterized by Neumann and Pati \cite{neumann2013reciprocal}, and the corresponding problem for weighted graphs was addressed by Panda and Pati \cite{panda2015inverse}. 
	The problem of characterizing weighted graphs with $SR$-property was considered by Bapat et al. \cite{bapat2016strong}.
	Until then, several families of graphs were known for which property $R$ and $SR$ are equivalent. However, Panda and Pati \cite{panda2016graphs} have shown that property $R$ and $SR$ are not equivalent in general by constructing a class of graphs that satisfies property $R$ but not $SR$.
	Barik et al. \cite{barik2022trees} proved that there are no non-trivial trees satisfying the weak reciprocal eigenvalue property. Subsequently, Barik et al. \cite{barik2024singular} constructed a class of connected weighted singular graphs whose non-zero spectrum satisfies $SR$-property.
	In addition, several other significant contributions in this direction can be found in \cite{panda2017some, bapat2017self, barik2024non, liu2025classes, kadu2024reciprocal, barik2021classes, kalita2014reciprocal}.
	
	Section \ref{sec_2} introduces the necessary preliminaries, including definitions of hypergraphs, hypermatrices, their standard products, and the concept of main eigenvalues for hypergraphs.
	Section \ref{sec_3} develops explicit expressions for the adjacency, Laplacian, and signless Laplacian eigenvalues corresponding to the join, Kronecker, and corona products of two uniform hypergraphs.
	The final section shows that power hypertrees do not satisfy the weak reciprocal eigenvalue property and further extends this notion by proposing a generalized definition of the reciprocal eigenvalue property for hypergraphs.
	\section{Preliminaries}\label{sec_2}
	
	A hypermatrix of order $r, r \geq 1$ is a generalization of the notion of a matrix, and is represented as a $n_1 \times n_2 \times \ldots \times n_r$ multi-array whose entries are denoted using the $r$-tuple of indices, $(i_1,\ldots, i_r),$ where $1 \leq i_j \leq n_j, ~ 1 \leq j \leq r.$    
	A hypermatrix $\mathcal{T}$ of dimension $n_1 \times \ldots \times n_t $ is said to be a square hypermatrix of dimension $n$ if each $n_i, ~ 1 \leq i \leq t$ is either equal to $n$ or equal to $1$.
	For example, a square matrix can be said to be a hypermatrix of dimension $1 \times n \times n,$ or $n \times 1 \times n,$ or $n \times n \times 1$. 
	In general, any hypermatrix of lower order can be viewed as a hypermatrix with prescribed coordinates of higher (expected) order in this manner. 
	This is exactly how the dimension of the row vector is written as $1 \times n$ and that of the column vector as $n \times 1. $ Strictly speaking, $\mathcal{T}$ is said to be a square hypermatrix of order $t$ and dimension $n$ if each $n_i = n, ~ 1 \leq i \leq t$.
	
	In 2005, the eigenvalues and eigenvectors for a square hypermatrix was defined by Qi \cite{qi2005eigenvalues} and the singular values of a non-square hypermatrix was defined by Lim \cite{lim2005singular}. 
	A pair $(\lambda, \mathbf{y})$ of a complex number and a complex vector of length $n$ is said to form an eigenpair for a square hypermatrix $\mathcal{T}$ of order $r$ and dimension $n$ if it satisfies the following condition:
	\begin{equation*}
		\mathcal{T} \mathbf{y}^{r-1} = \lambda \mathbf{y}^{[r-1]},
	\end{equation*}
	where $\mathcal{T} \mathbf{y}^{r-1} $ is a vector of length $n$ whose $j$-th entry is given by 
	\begin{equation*}
		(\mathcal{T} \mathbf{y}^{r-1})_j = \sum\limits_{j_2,\ldots, j_r=1}^n \mathcal{T}_{jj_2\ldots j_r}x_{j_2}\cdots x_{j_r}, 
	\end{equation*}
	and $\mathbf{y}^{[r-1]}$ is also a vector of length of $n$  whose $j$-th entry is given by  $(\mathbf{y}^{[r-1]})_j = y_j^{r-1}.$ 
	
	Formally, a hypergraph $\mathcal{H}$ is a triple $(\mathcal{V}, \mathcal{E}, \phi)$, where the elements of $\mathcal{V}$ are called the vertices and there exists an injective map $\phi$ from $\mathcal{E}$ to the power set of $\mathcal{V}$, which is denoted by $2^{\mathcal{V}*}$. To avoid confusion, we denote by $2^{\mathcal{V}*}$, the set of all non-empty subsets of the vertex set.
	For the sake of simplicity, one can define a hypergraph as an ordered pair $(\mathcal{V}, \mathcal{E})$ of vertices and hyperedges, where  $\mathcal{E}$ is a subset of  $2^{\mathcal{V}*}$ (that is, identifying $e$ and the image of $e \in \mathcal{E}$ under $\phi$). Throughout this article, we treat the elements of $\mathcal{E}$ as a subset of the vertex set.
	The degree of a vertex $i$ in a hypergraph $\mathcal{H}$ is the number of hyperedges of $\mathcal{H}$ containing $i$, and is denoted by $d_{\mathcal{H}}(i)$ or simply $d(i)$ if there is no confusion in the hypergraph under consideration. A path between two distinct vertices in a hypergraph is an alternating sequence of vertices and hyperedges that starts and ends with a vertex such that the consecutive elements in the sequence are incident, and all the elements in the sequence are distinct. If there exists a path between every two distinct vertices in a hypergraph, then we call it a connected hypergraph.
	
	The size of a hyperedge $e$ is the number of vertices in $e$, that is, $|e|$. If the size of all the hyperedges of the hypergraph are equal, say to $r$, then the hypergraph is said to be an $r$-uniform hypergraph. For an $r$-uniform hypergraph $\mathcal{H}$ on the vertex set $[n]:=\{1,2,\ldots,n\}$, the associated adjacency hypermatrix \cite{cooper2012spectra}, $\mathcal{A}=\mathcal{A}(\mathcal{H})$, is an order $r$, dimension $n$ hypermatrix whose $(i_1,\ldots,i_r)$-th entry is $\frac{1}{(r-1)!} $ if $\{ i_1, \ldots, i_r\}$ is a hyperedge of $\mathcal{H}$, and zero otherwise. 
	Suppose that $\mathcal{D}$ denotes the order $r$, dimension $n$ hypermatrix, whose $(i_1, \ldots, i_r)$-th entry is given by $d_{\mathcal{H}}(i_1)$ if $i_1=\cdots = i_r,$ and zero otherwise. The Laplacian and sign-Laplacian hypermatrix of $\mathcal{H}$ are defined as $\mathcal{D}-\mathcal{A}$ and $\mathcal{D} + \mathcal{A},$ respectively.
	

	\subsection{Standard Product of two Hypermatrices}
	
	Shao \cite{shao2013general} proposed a general product of two symmetric, dimension $n$, square hypermatrices $\mathcal{T}$ and $\mathcal{S}$ of varied orders. While this definition enables numerous applications, as explored in \cite{shao2013general} and subsequent works, it exhibits from several key limitations: under \emph{vector from the left}, it applies only when the first hypermatrix has order at least $2$, preventing left-multiplication of a compatible vector by a hypermatrix; under \emph{restricted times multiplications to a vector}, for a hypermatrix $\mathcal{A}$ of order $r \geq 3$ and vector $\mathbf{x}$, $\mathcal{A} \mathbf{x}^{r-1}$ is computable, but $\mathcal{A} \mathbf{x}^i$ cannot be formed for $1 \leq i \leq r-2$; under \emph{hidden modes}, it does not support multiplication by $i \leq r$ distinct vectors (or by any hypermatrix) across all $i$ modes of the hypermatrix; and under \emph{high order}, the product of two higher-order  hypermatrices yields a hypermatrix of very  high order.
	
	Let $A=(a_{i_1i_2})$ and $B=(b_{j_1j_2})$ be two square matrices of dimension $n$. 
	Can the product $A\cdot B$ be viewed as the multiplication of the matrix $A$ along its second coordinate to the matrix $B$ along its first coordinate? If so, then why not the multiplication of $A$ along its first coordinate to a matrix $B$ along its first coordinate, and all other possibilities? But these possibilities of the products are simplified using the transpose of the matrix. 
	Given $i, j \in \{ 1,2 \},$ suppose that $A \substack{^i * ^j} B$ denotes the multiplication of $A$ along its coordinate $i$ to the matrix $B$ along its coordinate $j$.
	Hence, it can be seen that $A\cdot B= A \substack{^2 * ^1} B,~ A^\intercal \cdot B = A \substack{^1 * ^1} B, ~ A\cdot B^\intercal= A \substack{{^2} * {^2}} B $ and $A^\intercal\cdot B^\intercal= A \substack{^1 * ^2} B.$
	
	The $r$-mode product of a hypermatrix and a matrix (or a vector) was introduced in \cite{bader2006algorithm}. Also, the $1$-mode product of two hypermatrices was defined in \cite{comon2002tensor}. In our proposed definition, we generalize this idea to the all-mode product of two hypermatrices. Let us start with defining the multiplication along the \emph{last} coordinate of the first hypermatrix to the first coordinate of the second hypermatrix.

	\begin{definition}\label{product_dot}
		Given an order $t$ hypermatrix $\mathcal{T}$ of dimension $n_1 \times \ldots n_{t}$, and an order $s$ hypermatrix $\mathcal{S}$ of dimension $m_1 \times \ldots \times m_s.$ If $n_t=m_1$, say equal to $n$, then the product $\mathcal{M} := \mathcal{T} \cdot \mathcal{S}$ is a hypermatrix of order $t+s-2,$ and dimension $n_1 \times \ldots \times n_{t-1} \times m_2 \times \ldots m_s, $ whose $(i_1,\ldots,i_{t+s-2})$-th entry is given by 
		\begin{equation*}
			\mathcal{M}_{i_1\ldots i_{t+s-2}} = \sum\limits_{k=1}^n \mathcal{T}_{i_1 \ldots i_{t-1}k} \mathcal{S}_{ki_{t} \ldots i_{t+s-2}}.
		\end{equation*}
	\end{definition}

	\begin{proposition}
		Suppose that $\mathcal{B}, \mathcal{T} $ and $ \mathcal{S} $ are the hypermatrices of dimensions $l_1 \times \ldots \times l_b, ~ n_1 \times \ldots \times n_t$ and $m_1 \times \ldots \times m_s$, respectively.
		If $b \geq 2,$  $l_b = n_1$ and $n_t = m_1,$ then 
		$$(\mathcal{B} \cdot \mathcal{T}) \cdot \mathcal{S} = \mathcal{B} \cdot (\mathcal{T} \cdot \mathcal{S}).$$
	\end{proposition}
	\begin{proof}
		Let us assume that $l_b = n_1=n$ and $n_t = m_1=m.$ By definition,
		\begin{align*}
			((\mathcal{B} \cdot \mathcal{T}) \cdot \mathcal{S})_{h_1\ldots h_{b+t+s-3}}
			&=  \sum\limits_{h=1}^m \left( \sum\limits_{k=1}^n \mathcal{B}_{h_1 \ldots h_{b-1}k} \mathcal{T}_{k h_b \ldots h_{b+t-3}h}  \right) \mathcal{S}_{h h_{b+t-2} \ldots h_{b+t+s-4}} \\
			&= \sum\limits_{k=1}^n \mathcal{B}_{h_1 \ldots h_{b-1}k} \left( \sum\limits_{h=1}^m \mathcal{T}_{k h_b \ldots h_{b+t-3}h}   \mathcal{S}_{h h_{b+t-2} \ldots h_{b+t+s-4}} \right)
			= (\mathcal{B} \cdot (\mathcal{T} \cdot \mathcal{S}))_{h_1\ldots h_{b+t+s-3}}.
		\end{align*}
	\end{proof}

	For an order $t \geq 2$ hypermatrix, $\mathcal{T}$ of dimension $n_1 \times n_2 \times \ldots \times n_t$, we can obtain an associated linear transformation, $\mathcal{F}_{\mathcal{T}}: \mathbb{C}^{n_t} \longrightarrow \mathbb{C}^{n_1 \times n_2 \times \ldots \times n_{t-1}}$ as, \begin{equation*} \mathcal{F}_{\mathcal{T}}(\mathbf{x}) := \mathcal{T} \cdot \mathbf{x} = x_1 \mathcal{T}^{(1)} + \cdots + x_{n_t} \mathcal{T}^{(n_t)}, \end{equation*} where $\mathcal{T}^{(j)}, 1 \leq j \leq n_t $ is the $j$-th slice of $\mathcal{T}$ along the $t$-th coordinate.  When $t=1$, $\mathcal{T}$ is a vector of dimension $n$; call it $\mathbf{y}$. Then $\mathcal{F}_{\mathbf{y}}$ is a linear mapping from $\mathbb{C}^n$ to $\mathbb{C}$ defined as $\mathcal{F}_{\mathbf{y}}(\mathbf{x}) := \mathbf{y} \cdot \mathbf{x} = \sum_{i=1}^n y_i x_i. $ Also, when $t=0$, $\mathcal{T}$ is a scalar, say $\alpha$. Then, $\mathcal{F}_{\alpha}$ is a linear transformation from $\mathbb{C}^n$ to $\mathbb{C}$ that is defined as $\mathcal{F}_{\alpha}(\mathbf{x}) := \alpha \cdot \mathbf{x} = \alpha \mathbf{x}.$ In fact, for any $\alpha \in \mathbb{C},$ $\alpha \cdot \mathcal{T}$ is defined to be the entry-wise product of $\alpha$ and $\mathcal{T}$, denoted by $\alpha \mathcal{T}. $ The product of two vectors $\mathbf{x}$ and $\mathbf{y}$ of dimension $n$ is well defined, that is, the dot product $\mathbf{x} \cdot \mathbf{y} = \sum_{i=1}^n x_iy_i$.
	
	Suppose that $\mathbf{T}_t$ denotes the collection of all square hypermatrices of order $t$ and dimension $n$. Let $\mathbf{T} := \bigcup_{t \in (\mathbb{Z}_+ \cup \{0\})} \mathbf{T}_t$ be the collection of all square hypermatrices of dimension $n$ (which properly contain $\mathbb{C}$ and $\mathbf{C}^n$ also). For $\alpha \in \mathbb{C} \subset \mathbf{T}$, we define $\alpha \cdot \mathcal{T}=\mathcal{T} \cdot \alpha := \alpha \mathcal{T}$, which is identical to the scalar multiplication. Note that $(\mathbf{T}, \cdot)$ forms a monoid with $1 \in \mathbb{C}$ as the identity element. We define a product of two linear transformations $\mathcal{F}_{\mathcal{T}}: \mathbb{C}^{n_t} \longrightarrow \mathbb{C}^{n_1 \times \ldots \times n_{t-1}} $ and $\mathcal{F}_{\mathcal{S}}: \mathbb{C}^{m_s} \longrightarrow \mathbb{C}^{m_1 \times \ldots \times m_{s-1}}$ as $\mathcal{F}_{\mathcal{T}} \circ \mathcal{F}_{\mathcal{S}} := \mathcal{F}_{\mathcal{T} \cdot \mathcal{S}},$ which is a map from $\mathbb{C}^{m_s}$ to $\mathbb{C}^{n_1 \times \ldots \times n_{t-1} \times m_2 \times \ldots m_{s-1}}.$ Again, note that $(\{\mathcal{F}_{\mathcal{T}}: \mathcal{T} \in \mathbf{T} \}, \circ)$ forms a monoid with $\mathcal{F}_1$ as the identity element. 
	
	Similar to the matrix case, we make use of the transpose of a hypermatrix with respect to a permutation to obtain the universal definition for the product of two hypermatrices. 
	
	\begin{definition}\cite{pan2014tensor}\label{hypermatrix_transpose}
		Consider an order $t$ hypermatrix $\mathcal{T}$ of dimension $n_1 \times \ldots \times n_t$. The transpose of $\mathcal{T}$ with respect to a permutation $() \neq \sigma \in S_t$, is a hypermatrix $\mathcal{S}$ of dimension  $n_{\sigma(1)} \times \ldots \times n_{\sigma(t)}$, where $\mathcal{S}_{k_{\sigma(1)} \ldots k_{\sigma(t)}} = \mathcal{T}_{k_{1}\ldots  k_{t}},$ and we denote this hypermatrix by $\mathcal{T}^{\sigma}$. 
	\end{definition}
	
	If $t=2$, then the only transpose possible is corresponding to the permutation $(1 ~ 2) \in S_2$, which is the usual matrix transpose. Hence a hypermatrix of order $t$ can have $t!-1$ transposes. Further properties and applications of this generalization of matrix transpose can be found in \cite{pan2014tensor}.
	
	\begin{definition}\label{product_star}
		Let $\mathcal{T}$ be an order $t$ hypermatrix of dimension $n_1 \times \ldots \times n_t$, and $\mathcal{S}$ be an order $s$ hypermatrix of dimension $m_1 \times \ldots \times m_s$. Suppose that $n_i = m_j$ (say both equal to $n$), then the multiplication of $\mathcal{T}$ along its $i$-th coordinate to $B$ along its $j$-th coordinate results in an order $t+s-2$ hypermatrix $\mathcal{N} := \mathcal{T} \substack{^i * ^j} \mathcal{S}$ of dimension $n_1 \times \ldots \times n_{i-1} \times n_{i+1} \times \ldots \times n_t \times m_1 \times \ldots \times m_{j-1} \times m_{j+1} \times \ldots \times m_s$ whose $(k_1, \ldots, k_{t+s-2})$-th entry is given by 
		\begin{equation*}
			\mathcal{N}_{k_1 \ldots k_{t+s-2}} = \sum\limits_{k=1}^n \mathcal{T}_{k_1 \ldots k_{i-1} k k_{i} \ldots k_{t-1}} \mathcal{S}_{k_{t} \ldots k_{t+j-2}k k_{t+j-1} \ldots k_{t+s-2}},
		\end{equation*}
		where 
		$k_p \in 
		\begin{cases} [n_p] & \text{ if } 
			1 \leq p \leq i-1,\\
			[n_{p+1}] & \text{ if }  i \leq p \leq t-1,
		\end{cases} $ 
		and 
		$k_q \in 
		\begin{cases} 
			[m_{q-t+1}] & \text{ if }  t \leq q \leq t+j-2,\\
			[m_{q-t+2}] & \text{ if } t+j-1 \leq q \leq t+s-2.
		\end{cases} $
	\end{definition}

	\begin{remark}
		Definition \ref{product_dot} can be viewed as the \emph{standard product} of two hypermatrices, since the product of two hypermatrices along any of their coordinates can be reduced to Definition \ref{product_dot} using Definition \ref{hypermatrix_transpose}. That is, for an order $t$ hypermatrix $\mathcal{T}$ of dimension $n_1 \times \ldots \times n_{t}$, and an order $s$ hypermatrix $\mathcal{S}$ of dimension $m_1 \times \ldots \times m_s$. If $n_i=m_j$, then 
		\begin{equation*}           
			\mathcal{T} \substack{^i * ^j} \mathcal{S} =\mathcal{T}^{\sigma_i} \cdot \mathcal{S}^{\pi_j}, 
		\end{equation*}       
		where $\sigma_i = (i ~ t) \in S_t$ and $\pi_j = (1 ~ j) \in S_s.$
	\end{remark}

	Under the assumption that the entries of the hypermatrices are from the field, there are totally $ts$ distinct modes of products possible between a hypermatrix of order $t$ and a hypermatrix of order $s$.

	\subsection{Main Eigenvalues of a Hypermatrix}
	
	In this section, we define the main/non-main eigenvalues of the hypermatrix, and we obtain some main/non-main eigenvalues of the hypermatrices associated with the hypergraph.
	
	\begin{definition}
		Suppose that $\mathcal{M}$ is an $r$-uniform, dimension $n$ hypermatrix. We call an eigenvalue $\mu$ of a hypermatrix $\mathcal{M}$, a non-main eigenvalue if there exist an eigenvector of $\mathbf{x}$ of $\mathcal{M}$ associated with $\mu$ such that 
		\begin{equation}\label{non-main_condition}
			\sum\limits_{S \subset [n], |S|=r-1} \mathbf{x}^S = 0.
		\end{equation}
		In other words, $\mathbf{x}$ is called a non-main eigenvector.
		Otherwise (if there does not exist such $\mathbf{x}$) it is called a main eigenvalue of $\mathcal{M}$. 
	\end{definition}
	
	\begin{remark}
		Alternatively, an eigenvector $\mathbf{x}$ of an order $r \geq 2$, dimension $n \geq 2$ hypermatrix is said to be non-main if $\mathcal{T} \mathbf{x}^{r-1} = 0$,
		where $\mathcal{T}$ is a $(0,1)$-adjacency hypermatrix of $(r-1)$-uniform complete hypergraph on $n$ vertices.
		
		\begin{itemize}
			\item In the trivial case, if $r=2$, then the adjacency hypermatrix of $1$-uniform hypergraph is all one vector of length $n$ (i.e., an order $1$, dimension $n$ hypermatrix).
			\item As a non-trivial case, if $r=3$, then $\mathcal{T}=J-I=A(K_n)$ is the adjacency matrix of the complete graph on $n$ vertices. Now, an eigenvector $\mathbf{x}$ of an order $3$, dimension $n$ hypermatrix is said to be a non-main eigenvector if $\mathbf{x}^\intercal \mathcal{T} \mathbf{x} = \mathbf{x}^\intercal J \mathbf{x} - \mathbf{x}^\intercal I \mathbf{x} =0. $
			\item  As mentioned earlier, for $r \geq 3$, $\mathcal{T} := (r-2)! \mathcal{A}(\mathcal{K}_n^{(r-1)})$.
		\end{itemize}
	\end{remark}
	
	For $r\geq 3$, one can note that there always exist $n$ non-main eigenvectors corresponding to the eigenvalue $0$ of any $r$-uniform hypergraph (canonical vectors are the non-main eigenvectors), and the spectral radius is always a main eigenvalue of any non-trivial $r$-uniform hypergraph. 
	
	\begin{proposition}
		Let $\mathcal{K}_n^{(r)}$ be the $r$-uniform ($r \geq 3$) complete hypergraph on the vertex set $[n]$, $n \geq r$.
		\begin{itemize}
			\item Suppose that  $(\lambda, \mathbf{y})$ is an eigenpair of $\mathcal{A}(\mathcal{K}_n^{(r)})$. Then, $\lambda \neq 0$ is a non-main eigenvalue if and only if $\sum_{i=1}^n y_i^{r-1}=0.$ 
			Furthermore,      $0$ is always a non-main eigenvalue of $\mathcal{A}(\mathcal{K}_n^{(r)})$.
			\item Suppose that  $(\mu, \mathbf{z})$ (resp. $(\beta, \mathbf{x})$) is an eigenpair of $\mathcal{L}(\mathcal{K}_n^{(r)})$ (resp. $\mathcal{Q}(\mathcal{K}_n^{(r)})$). Then, $\mu \neq \binom{n-1}{r-1}$ is a non-main eigenvalue of $\mathcal{L}(\mathcal{K}_n^{(r)})$ (resp. $\mathcal{Q}(\mathcal{K}_n^{(r)})$) if and only if $\sum_{i=1}^n z_i^{r-1}=0$ (resp. $\sum_{i=1}^n x_i^{r-1}=0$). 
			Furthermore, $\binom{n-1}{r-1}$ is always a non-main eigenvalue of $\mathcal{L}(\mathcal{K}_n^{(r)})$ and $\mathcal{Q}(\mathcal{K}_n^{(r)})$. 
		\end{itemize}
	\end{proposition}
	\begin{proof}
		By using the eigen-equation, if $(\lambda, \mathbf{y})$ forms an eigenpair for $\mathcal{A}=\mathcal{A}(\mathcal{K}_n^{(r)})$, then $\lambda y_i^{r-1} = \sum_{e \in \mathcal{E}_i(\mathcal{H})} \mathbf{y}^{e \setminus \{i\}}$, for each $i, ~1 \leq i \leq n$. 
		If $\lambda=0$, then $$0= \lambda \sum\limits_{i=1}^n y_i^{r-1} = \sum_{i=1}^n \sum_{e \in \mathcal{E}_i(\mathcal{H})} \mathbf{y}^{e \setminus \{i\}} = (n-r+1)\sum\limits_{S \subseteq [n], |S|=r-1} \mathbf{y}^S.$$ Since $n \geq r$, we have $\mathbf{y}$ is a non-main eigenvector.
		If $\lambda \neq 0,$ then by the definition, $\mathbf{y}$ is a  non-main eigenvector if and only if $\sum_{i=1}^n y_i^{r-1}=0.$ 
		
		Similar argument using the eigen-equations of $\mathcal{L}(\mathcal{K}_n^{(r)})$ and $\mathcal{Q}(\mathcal{K}_n^{(r)})$ can be used to show that $\binom{n-1}{r-1}$ is a non-main eigenvalue of the both.
	\end{proof}

	\section{Eigenvalues of Hypergraph Products}\label{sec_3}
	
	This section is dedicated to the computation of the adjacency, Laplacian and signless-Laplacian eigenvalues of the join, Kronecker product, and corona product of two hypergraphs in terms of the eigenvalues of the underlying graphs.
	
	\subsection{Join}
	
	\begin{definition}\cite{shetty2025forgotten}
		Let $\mathcal{H}_1$ and $\mathcal{H}_2$ be two $r$-uniform hypergraphs on the vertex sets $\mathcal{V}(\mathcal{H}_1)$ and $\mathcal{V}(\mathcal{H}_2)$, with edge sets $\mathcal{E}(\mathcal{H}_1)$ and $\mathcal{E}(\mathcal{H}_2),$ respectively. Then ($r$-uniform) join of $\mathcal{H}_1$ and $\mathcal{H}_2$ is denoted by $\mathcal{H}_1 \vee \mathcal{H}_2$ is an $r$-uniform hypergraph on the vertex set $\mathcal{V}(\mathcal{H}_1 \vee \mathcal{H}_2) := \mathcal{V}(\mathcal{H}_1) \cup \mathcal{V}(\mathcal{H}_2)$, and the edge set is given by $\mathcal{E}(\mathcal{H}_1 \vee \mathcal{H}_2):=\mathcal{E}(\mathcal{H}_1) \cup \mathcal{E}(\mathcal{H}_1) \cup \mathcal{E}^*$, where 
		$$\mathcal{E}^* = \{ e \subseteq \mathcal{V}(\mathcal{H}_1 \vee  \mathcal{H}_2) \vert e \cap  \mathcal{V}(\mathcal{H}_i) \neq \emptyset, i= 1,2, \text{ and } |e|=r \}.$$
	\end{definition}

	\begin{theorem}
		Let $\mathcal{H}_1$ and $\mathcal{H}_2$ be two $r$-uniform hypergraphs of order $n_1$ and $n_2$, respectively. Suppose that $\mathcal{H}:=\mathcal{H}_1 \vee \mathcal{H}_2$ denotes the join of $\mathcal{H}_1$ and $\mathcal{H}_2$. 
		If $\eta$ is a non-main eigenvalue of $\mathcal{L}(\mathcal{H}_1)$ and $\zeta$ is a non-main eigenvalue of $\mathcal{L}(\mathcal{H}_2)$, then $\eta$ and $\zeta$ are the non-main eigenvalues of $\mathcal{L}(\mathcal{H}).$
		Furthermore, if $\epsilon_i \neq 0, 1 \leq i \leq 2r-2$ are the solutions of the equation
		\begin{equation}
			\sum\limits_{j=1}^{r-1} \binom{n_2}{j}\binom{n_1}{r-1-j} (1-\epsilon^{-j}) - \sum\limits_{j=1}^{r-1} \binom{n_1}{j}\binom{n_2}{r-1-j} (1-\epsilon^{j}) = 0, \label{eqn_join_laplacian}
		\end{equation} then $P_i:= \sum\limits_{j=1}^{r-1} \binom{n_2}{j}\binom{n_1}{r-1-j} (1-\epsilon_i^{-j})=\sum\limits_{j=1}^{r-1} \binom{n_1}{j}\binom{n_2}{r-1-j} (1-\epsilon_i^{j})$ are the eigenvalues of $\mathcal{L}(\mathcal{H}).$
	\end{theorem}
	\begin{proof}
		Without loss of generality, assume that $[n_1+n_2]$ be the vertex set of $\mathcal{H}_1 \vee \mathcal{H}_2$, and let $[n_1]$ be the vertices corresponding to $\mathcal{V}(\mathcal{H}_1)$ and $[n_1 + n_2] \setminus [n_1]$ be the vertices corresponding to $\mathcal{V}(\mathcal{H}_2)$. 
		Suppose that $\eta$ is a non-main eigenvalue of $\mathcal{L}(\mathcal{H}_1)$ corresponding to a non-main eigenvector $\mathbf{x}$, and $\zeta$ is a non-main eigenvalue of $\mathcal{L}(\mathcal{H}_2)$ corresponding to a non-main eigenvector $\mathbf{z}$. 
		Then by using the definition of non-main eigenvector, it is direct to verify that $\eta$ is a non-main eigenvalue of $\mathcal{H}_1 \vee \mathcal{H}_2$ and corresponding non-main eigenvector is 
		$\left[\begin{matrix}
			\mathbf{x}\\
			\mathbf{0}
		\end{matrix} \right]$.
		Similarly,  $\zeta$ is a non-main eigenvalue of $\mathcal{H}_1 \vee \mathcal{H}_2$ and corresponding non-main eigenvector is 
		$\left[\begin{matrix}
			\mathbf{0}\\
			\mathbf{z}
		\end{matrix} \right]$.
		Let us define $\mathbf{y}=\mathbf{y}^{(i)}:=
		\left[ \begin{matrix}
			\epsilon_i \mathbf{1}_{n_1} \\
			\mathbf{1}_{n_2}
		\end{matrix} \right],$ where $\epsilon_i$ is a solution of the Equation \ref{eqn_join_laplacian}, and $P_i := \sum\limits_{j=1}^{r-1} \binom{n_2}{j}\binom{n_1}{r-1-j} (1-\epsilon_i^{-j})$. 
		Then for any $j \in [n_1],$ we have by eigen-equation
		\begin{align*}
			(\mathcal{L}(\mathcal{H}) \mathbf{y}^{r-1})_j - P_i y_j^{r-1} 
			&= d_{\mathcal{H}}(j) y_j^{r-1} - \sum\limits_{e \in \mathcal{E}_j(\mathcal{H})} \mathbf{y}^{e \setminus \{j \}}  - P_i y_j^{r-1}\\
			&=\left(\sum\limits_{j=1}^{r-1} \binom{n_2}{j}\binom{n_1}{r-1-j} - P_i \right)\epsilon_i^{r-1} - \sum\limits_{j=1}^{r-1} \binom{n_2}{j} \binom{n_1}{r-1-j} \epsilon_i^{r-1-j} = 0. 
		\end{align*}
		Using Equation \ref{eqn_join_laplacian}, let us assign $P_i:=\sum\limits_{j=1}^{r-1} \binom{n_1}{j}\binom{n_2}{r-1-j} (1-\epsilon_i^{j}).$ Hence, for any $k \in [n_1+n_2]\setminus [n_1]$, we have
		\begin{align*}
			(\mathcal{L}(\mathcal{H})\mathbf{y}^{r-1})_k - P_iy_k^{r-1} &= d_{\mathcal{H}}(k) y_k^{r-1} - \sum\limits_{ e \in \mathcal{E}_k(\mathcal{H})} \mathbf{y}^{e \setminus \{k \} } - P_i y_k^{r-1} \\
			&=\sum\limits_{j=1}^{r-1} \binom{n_1}{j} \binom{n_2}{r-1-j} - P_i - \sum\limits_{j=1}^{r-1} \binom{n_1}{j} \binom{n_2}{r-1-j} \epsilon_i^j=0. 
		\end{align*}
	\end{proof}

	
	\begin{theorem}
		Given two $r$-uniform hypergraphs  $\mathcal{H}_1$ and $\mathcal{H}_2$ of orders $n_1$ and $n_2$, respectively.
		Let $\mathcal{H}:=\mathcal{H}_1 \vee \mathcal{H}_2$ be the $r$-uniform join of $\mathcal{H}_1$ and $\mathcal{H}_2$.
		\begin{itemize}
			\item If  $\lambda$ is a non-main eigenvalue of $\mathcal{A}({\mathcal{H}_1})$ and $\mu$ is a non-main eigenvalue of $\mathcal{A}({\mathcal{H}_2}),$ then $\lambda$ and $\mu$ are the non-main eigenvalues of $\mathcal{A}(\mathcal{H})$.
			Moreover, suppose that $\mathcal{H}_1$ is $d_1$ regular, $\mathcal{H}_2$ is $d_2$ regular, and if  $ \epsilon_i \neq 0, 1 \leq i \leq 2r-2$ are the solutions of \begin{equation*}
				d_1-d_2 + \sum\limits_{j=1}^{r-1} \binom{n_2}{j}\binom{n_1}{r-1-j} \epsilon^{-j} -\sum\limits_{j=1}^{r-1} \binom{n_1}{j}\binom{n_2}{r-1-j} \epsilon^{j}=0,     \label{eqn_join_adjacency}    
			\end{equation*} 
			then $P_i:=d_1+ \sum\limits_{j=1}^{r-1} \binom{n_2}{j}\binom{n_1}{r-1-j} \epsilon_i^{-j} = 0 $ is an eigenvalue of $\mathcal{A}(\mathcal{H})$.\\
			Furthermore, if $\epsilon_i$ is positive, then $P_i$ is the spectral radius of $\mathcal{A}(\mathcal{H})$.
			\item If  $\theta$ is  non-main eigenvalue of $\mathcal{Q}({\mathcal{H}_1})$ and $\beta$ is a non-main eigenvalue of $\mathcal{Q}({\mathcal{H}_2}),$ then $\theta$ and $\beta$ are the non-main eigenvalues of $\mathcal{Q}(\mathcal{H})$. 
			Moreover, suppose that $\mathcal{H}_1$ is $d_1$ regular, $\mathcal{H}_2$ is $d_2$ regular, and if  $ \epsilon_i \neq 0, 1 \leq i \leq 2r-2$ are the solutions of 
			\begin{equation*}
				2(d_1-d_2) + \sum\limits_{j=1}^{r-1} \binom{n_2}{j}\binom{n_1}{r-1-j} (1+\epsilon^{-j}) -\sum\limits_{j=1}^{r-1} \binom{n_1}{j}\binom{n_2}{r-1-j} (1+\epsilon^{j})=0, \label{eqn_join_signless} 
			\end{equation*} then $Q_i:= 2d_1+ \sum\limits_{j=1}^{r-1} \binom{n_2}{j}\binom{n_1}{r-1-j} (1+\epsilon_i^{-j}) =  2d_1+ \sum\limits_{j=1}^{r-1} \binom{n_1}{j}\binom{n_2}{r-1-j} (1+\epsilon_i^{j})$ is an eigenvalue of $\mathcal{Q}(\mathcal{H})$.\\
			Furthermore, if $\epsilon_i$ is positive, then $Q_i$ is the spectral radius of $\mathcal{Q}(\mathcal{H})$.
		\end{itemize}  
	\end{theorem}
	\begin{proof}
		Without loss of generality, $[n_1+n_2]$ be the vertex set of $\mathcal{H}=\mathcal{H}_1 \vee \mathcal{H}_2$ with
		$[n_1]$ be the vertices corresponding to $\mathcal{V}(\mathcal{H}_1)$, and $[n_1+n_2] \setminus [n_1]$ be the vertices corresponding to $\mathcal{V}(\mathcal{H}_2)$.
		Suppose that $\lambda$ and $\mu$ are the non-main eigenvalues of $\mathcal{A}(\mathcal{H}_1)$ and $\mathcal{A}(\mathcal{H}_2)$ with corresponding non-main eigenvectors $\mathbf{x}$ and $\mathbf{z}$, respectively. Then, by using the definition of non-main eigenvector, it is direct to verify that $\lambda$ and $\mu$ are the non-main eigenvalues of $\mathcal{A}(\mathcal{H})$ with corresponding non-main eigenvectors $ \left[ \begin{matrix}
			\mathbf{x}\\
			\mathbf{0}_{n_2} 
		\end{matrix} \right]$ and $ \left[ \begin{matrix}
			\mathbf{0}_{n_1}\\
			\mathbf{z} 
		\end{matrix} \right],$ respectively.
		If $\mathcal{H}_1$ is $d_1$-regular and $\mathcal{H}_2$ is $d_2$-regular, then define $ \mathbf{y}=\mathbf{y}^{(i)} := \left[ \begin{matrix}
			\epsilon_i \mathbf{1}_{n_1} \\
			\mathbf{1}_{n_2}
		\end{matrix} \right], $ and $P_i := d_1+ \sum\limits_{j=1}^{r-1} \binom{n_2}{j} \binom{n_1}{r-1-j} \epsilon_i^{-j}$. By using the eigen-equation, for any $j \in [n_1],$ we have
		\begin{align*}
			(\mathcal{A}(\mathcal{H}) \mathbf{y}^{r-1})_j - P_i y_j^{r-1} = \sum\limits_{e \in \mathcal{E}_j(\mathcal{H})} \mathbf{y}^{e\setminus \{ j\} } - P_i y_j^{r-1} = (d_1 -P_i) \epsilon^{r-1} + \sum\limits_{j=1}^{r-1} \binom{n_2}{j} \binom{n_1}{r-1-j} \epsilon^{r-1-j} =0.
		\end{align*}
		Again, by using Equation \ref{eqn_join_adjacency}, let $P_i:= d_2+ \sum\limits_{j=1}^{r-1} \binom{n_1}{j}\binom{n_2}{r-1-j} \epsilon_i^{j} $ for any $k \in [n_1+n_2] \setminus [n_1]$, we have 
		\begin{align*}
			(\mathcal{A}(\mathcal{H}) \mathbf{y}^{r-1})_k -  P_i y_j^{r-1} &= \sum\limits_{e \in \mathcal{E}_j(\mathcal{H})} \mathbf{y}^{e\setminus \{ j\} } - P_i y_j^{r-1} = (d_2 -P_i)  + \sum\limits_{j=1}^{r-1} \binom{n_1}{j} \binom{n_2}{r-1-j} \epsilon^{j} =0.
		\end{align*}
		If $\epsilon_i > 0$ for some $i$, then $\mathbf{y}^{(i)} >0$. Hence by using Perron-Frobenius theorem, $P_i$ is the spectral radius of $\mathcal{A}(\mathcal{H})$. 
		
		\noindent  The proof follows analogously for the signless Laplacian hypermatrix of uniform hypergraphs.
	\end{proof}

	\subsection{Kronecker Product}

	\begin{definition}\label{Kronecker_product}
		Given two $r$-uniform hypergraphs $\mathcal{H}_1$ and $\mathcal{H}_2$ of orders $n_1$ and $n_2$, respectively. 
		Then Kronecker product of $\mathcal{H}_1$ and $\mathcal{H}_2$, denoted by $\mathcal{H}_1 \otimes \mathcal{H}_2$ has the vertex set $\mathcal{V}(\mathcal{H}_1 \otimes \mathcal{H}_2)= \mathcal{V}(\mathcal{H}_1) \times \mathcal{V}(\mathcal{H}_2)$, and the edge set is given by $\mathcal{E}(\mathcal{H}_1 \otimes \mathcal{H}_2) = \{ \{ (u_1,v_1'), \ldots, (u_r,v_r') \} \vert e=\{u_1,\ldots,u_r \} \in \mathcal{E}(\mathcal{H}_1), e'=\{u_1',\ldots,u_r' \} \in \mathcal{E}(\mathcal{H}_2)\}$.
	\end{definition}
	Note that the number of vertices in $\mathcal{H}_1 \otimes \mathcal{H}_2$ is $n_1n_2$, and the number of hyperedges in $\mathcal{H}_1 \otimes \mathcal{H}_2$ is $(r-1)!m_1m_2$, where $m_1=\vert \mathcal{E}(\mathcal{H}_1)  \vert $ and $m_2=\vert \mathcal{E}(\mathcal{H}_2) \vert.$

	\begin{theorem}\cite{shao2013general}\label{Kronecker_adjacency}
		Given two $r$-uniform hypergraphs $\mathcal{H}_1$ and $\mathcal{H}_2,$ let $\mathcal{H}:=\mathcal{H}_1 \otimes \mathcal{H}_2$ be the Kronecker product of $\mathcal{H}_1$ and $\mathcal{H}_2$.
		If $(\lambda, \mathbf{x})$ is an eigenpair of $\mathcal{H}_1$ and $(\mu,\mathbf{z})$ is an eigenpair of $\mathcal{H}_2$, then $((r-1)!\lambda\mu, \mathbf{x} \otimes \mathbf{z})$ forms an eigenpair for $\mathcal{H}$.
	\end{theorem}
	
	\begin{remark}
		If $r=2,$ then Kronecker (direct) product can be used to construct graphs satisfying reciprocal eigenvalue property from the smaller graphs. However, for $r \geq 3$, Theorem \ref{Kronecker_adjacency} shows that the Kronecker product no longer provides a direct solution.
	\end{remark}
	
	In the following result, we determine $n(r-1)^{n_1+n_2-2}$ Laplacian and signless Laplacian eigenvalues of Kronecker product of regular hypergraphs $\mathcal{H}_1 \otimes \mathcal{H}_2$, a generalization of the result from spectral graph theory.

	\begin{theorem}
		Given two $r$-uniform regular hypergraphs $\mathcal{H}_1$ and $\mathcal{H}_2$ with regularities $d_1$ and $d_2$, respectively.  Let $\mathcal{H}:=\mathcal{H}_1 \otimes \mathcal{H}_2$ be the Kronecker product of $\mathcal{H}_1$ and $\mathcal{H}_2$. 
		\begin{itemize}
			\item If $(\lambda, \mathbf{x})$ is an eigenpair of $\mathcal{L}(\mathcal{H}_1)$ and $(\mu, \mathbf{z})$ is an eigenpair of $\mathcal{L}(\mathcal{H}_2),$ then $(r-1)![\lambda d_2+ \mu d_1 - \lambda \mu]$ is an eigenvalue of $\mathcal{L}(\mathcal{H})$ with corresponding eigenvector $\mathbf{x} \otimes \mathbf{z}.$ 
			\item If $(\beta, \mathbf{w})$ is an eigenpair of $\mathcal{Q}(\mathcal{H}_1)$ and $(\theta, \mathbf{y})$ is an eigenpair of $\mathcal{Q}(\mathcal{H}_2),$ then $(r-1)![\beta \theta + 2d_1d_2 -(\beta d_2+ \theta d_1)]$ is an eigenvalue of $\mathcal{Q}(\mathcal{H})$ with corresponding eigenvector $\mathbf{w} \otimes \mathbf{y}.$
		\end{itemize}
	\end{theorem}
	\begin{proof}
		Let $ij':=(i, j'), ~1 \leq i \leq n_1, ~ 1 \leq j' \leq n_2 $ be the vertices of $\mathcal{H}_1 \otimes \mathcal{H}_2$, where $n_1$ and $n_2$ are the orders of $\mathcal{H}_1$ and $\mathcal{H}_2$, respectively. 
		Assume that $\lambda$ (resp. $\mu$) is an eigenvalue of $\mathcal{L}(\mathcal{H}_1)$ (resp $\mathcal{L}(\mathcal{H}_2)$) corresponding to the eigenvector $\mathbf{x}$ (resp. $\mathbf{z}$). 
		Let us define $\mathbf{y}:= \mathbf{x}\otimes \mathbf{z}$ and $\eta := (r-1)![\lambda d_2+ \mu d_1 - \lambda \mu]$. Then, for any $ij' \in \mathcal{V}(\mathcal{H}),$ 
		\begin{align*}
			(\mathcal{L}(\mathcal{H})\mathbf{y}^{r-1})_{ij'} - \eta y_{ij'}^{r-1}&= d_{\mathcal{H}}(ij') \mathbf{y}_{ij'}^{r-1} -\sum\limits_{e^* \in \mathcal{E}_{ij'}(\mathcal{H})} \mathbf{y}^{e^*\setminus \{ ij' \} }- \eta y_{ij'}^{r-1} \\
			&=(r-1)!d_1d_2x_i^{r-1}z_{j'}^{r-1}-\sum\limits_{e \in \mathcal{E}(\mathcal{H}_1)} \sum\limits_{e' \in \mathcal{E}(\mathcal{H}_2)} (r-1)! \mathbf{x}^{e\setminus \{ i \}} \mathbf{z}^{e'\setminus \{ j' \}} - \eta x_i^{r-1} z_{j'}^{r-1} = 0.          \end{align*}
		The proof follows analogously for the signless Laplacian hypermatrix of uniform hypergraphs.
	\end{proof}

	\subsection{Corona of two Hypergraphs}
	
	\begin{definition}\cite{shetty2025forgotten}
		Let $\mathcal{G}=(\mathcal{U}=\{ u_1, \ldots, u_{n_1}\}, \mathcal{E}')$ and $\mathcal{H}=(\mathcal{V}=\{ v_1, \ldots, v_{n_2}\}, \mathcal{E})$ be two $r$-uniform hypergraphs of order $n_1$ and $n_2$ respectively, also $ \{ \mathcal{H}^{(i)}=(\mathcal{V}_i=\{ v_1^{(i)}, \ldots, v_{n_2}^{(i)}\}, \mathcal{E}_i): 1 \leq i \leq n_1 \}$ be the collection of $n_1$ copies of $\mathcal{H}$
		and $\mathcal{W}_i = \mathcal{V}_i \cup \{ u_i \}$.
		Now, $r$-uniform corona product  $\mathcal{G} \circ^r \mathcal{H}$ of two hypergraphs $\mathcal{G}$ and $\mathcal{H}$ has the vertex set $\bigcup\limits_{i=1}^{n_1} \mathcal{W}_i$ and the edge set $\bigcup\limits_{i=1}^{n_1} \left( \mathcal{E}_i \cup \mathcal{E}_i^{\dagger} \right) \cup \mathcal{E}',$ where
		$$ \mathcal{E}_i^\dagger = \left\{e \subseteq \mathcal{W}_i: u_i \in e
		\text{ and } \vert e \vert = r  \right\}. $$ 
	\end{definition}

	\begin{theorem}\label{Corona_one_edge}
		Given an $r$-uniform hypergraph $\mathcal{H}_1$ on $n$ vertices, let $ \mathcal{H} = \mathcal{H}_1 \odot (r-1) K_1$. 
		\begin{itemize}
			\item  If $\lambda$ is an eigenvalue of $\mathcal{A}(\mathcal{H}_1)$, 
			then $\mu_t$ are the eigenvalues of $\mathcal{A}(\mathcal{H}),$
			where $\mu_t,1\leq t \leq r$ are the roots of the polynomial $$\mu^r - \lambda \mu^{r-1} -1.$$
			\item  If $\theta$ is an eigenvalue of $\mathcal{L}(\mathcal{H}_1)$,
			then $\mu_t \neq 1$ are the eigenvalues of $\mathcal{L}(\mathcal{H}),$ where $\mu_t, 1 \leq t \leq r$ are the roots of the polynomial 
			$$ (\mu-\theta-1)(1-\mu)^{r-1}+1.$$
			\item If $\beta$ is an eigenvalue of $\mathcal{Q}(\mathcal{H}_1),$ 
			then $\mu_t \neq 1$ are the eigenvalues of $\mathcal{Q}(\mathcal{H}),$ where $\mu_t, 1 \leq t \leq r$ are the roots of the polynomial     $$ (\mu-\beta-1)(\mu-1)^{r-1}-1.$$
		\end{itemize}
	\end{theorem}
	\begin{proof}
		Let $\mathcal{H}_1$ be a $r$-uniform hypergraph on  the vertex set $[n],$ and $H = \mathcal{H}_1 \odot (r-1) K_1$ be an $r$-uniform hypergraph with $p_i :=\{ i, n+1+(i-1)(r-1), \ldots, n+i(r-1) \}, 1 \leq i \leq n$ be the newly added hyperedges of $\mathcal{H}$.
		Suppose that $(\lambda \neq 0, \mathbf{x})$ is an eigenpair of $\mathcal{H}_1.$ 
		If $\mu_t, 1 \leq t \leq r$ denotes the roots of the polynomial $\mu^r-\lambda \mu^{r-1} -1,$ then define the vectors $\mathbf{y}=\mathbf{y}^{(i)} := [\mathbf{x}^\intercal, \ \frac{x_1}{\mu_t}\mathbf{1}_{r-1}^\intercal, \ldots  \frac{x_n}{\mu_t}\mathbf{1}_{r-1}^\intercal]^\intercal$,  $ 1 \leq t \leq r.$
		By using the eigen equation and the  of $\mu_t$, for any $i \in [n],$ we have
		\begin{equation*}
			\mu_t{y}_i^{r-1} -  (\mathcal{A}(\mathcal{H})\mathbf{y}^{r-1})_i = \mu_t{x}_i^{r-1} - (\mathcal{A}(\mathcal{H}_1)\mathbf{x}^{r-1})_i - \mathbf{y}^{p_i \setminus \{i\}} = \mu_t{x}_i^{r-1}-\lambda x_i^{r-1} - \left(\frac{x_i}{\mu_t}\right)^{r-1}=0.
		\end{equation*}
		For all $j \in p_i \setminus \{i\}$, we have 
		\begin{equation*}
			\mu_t y_j^{r-1}-(\mathcal{A}(\mathcal{H})\mathbf{y}^{r-1})_j =\mu_t \left(\frac{x_i}{\mu_t} \right)^{r-1} - \mathbf{y}^{p_i \setminus\{ j \} } = 0.
		\end{equation*}
		
		Similarly, suppose that $(\theta, \mathbf{w})$ is an eigenpair of $\mathcal{L}(\mathcal{H})$, and $\mu = \mu_i ( \neq 1 )$ is a root of the  polynomial  \\ $\left( \mu-\theta-1 \right)\left( 1-\mu \right)^{r-1}-1$. 
		Define $\mathbf{y}=\mathbf{y}^{(t)} := \left[ \mathbf{w}^\intercal, \frac{w_1}{1-\mu_t}\mathbf{1}_{r-1}^{\intercal}, \ldots , \frac{w_n}{1-\mu_t }\mathbf{1}_{r-1}^{\intercal} \right]^\intercal$.
		It is direct to show that $(\mu_t, \mathbf{y}^{(t)})$ forms an eigenpair for $\mathcal{L}(\mathcal{H})$.
		
		Also, suppose that $(\beta, \mathbf{z})$ is an eigenpair of $\mathcal{Q}(\mathcal{H})$, and $\mu = \mu_t ( \neq 1 )$ is a root of the 
		polynomial  $\left( \mu-\beta-1 \right)\left( \mu-1 \right)^{r-1}-1$. 
		Define $\mathbf{y}=\mathbf{y}^{(t)} := \left[ \mathbf{z}^\intercal, \frac{z_1}{\mu_t-1}\mathbf{1}_{r-1}^{\intercal}, \ldots , \frac{z_n}{\mu_t-1 }\mathbf{1}_{r-1}^{\intercal} \right]^\intercal$.
		Again, one can show that $(\mu_i, \mathbf{y}^{(i)})$ forms an eigenpair for $\mathcal{Q}(\mathcal{H})$.
	\end{proof}
	
	\begin{corollary}\cite{barik2007spectrum}
		Given a graph $\mathcal{G}$ on $n$ vertices, let $ \mathcal{G}_1 = \mathcal{G} \odot K_1$. 
		If $\mu$ is an eigenvalue of $\mathcal{G}_1$, then $\frac{-1}{\mu}$ is also an eigenvalue of $\mathcal{G}_1$.
	\end{corollary}

	We can see the notion of corona of a hypergraph $\mathcal{H}$ with $r-1$ copies of $K_1$ as adding a pendant hyperedge at each vertex of $\mathcal{H}$. In this point of view, we extend this and will define a hypergraph $\hat{\mathcal{H}}(k), ~ k \geq 1$ to be the hypergraph obtained from $\mathcal{H}$ by adding $k$ pendant hyperedges at each vertex of $\mathcal{H}$.
	Observe that $\hat{\mathcal{H}}(1) \cong \mathcal{H} \odot (r-1) K_1.$ 
	\begin{theorem}
		Given an $r$-uniform hypergraph $\mathcal{H}$ on $n$ vertices, let $\hat{\mathcal{H}}(k), k \geq 1$ be obtained from $\mathcal{H}$ by adding $k$ pendent hyperedges at each vertex of $\mathcal{H}$. Then
		\begin{itemize}
			\item  If $\lambda$ is an eigenvalue of $\mathcal{A}(\mathcal{H})$, 
			then $\mu_i$ are the eigenvalues of $\mathcal{A}(\hat{\mathcal{H}}),$
			where $\mu_i,1\leq i \leq r$ are the roots of the polynomial $$\mu^r - \lambda \mu^{r-1} -k.$$
			\item  If $\theta$ is an eigenvalue of $\mathcal{L}(\mathcal{H})$, then $\mu_i \neq 1$ are the eigenvalues of $\mathcal{L}(\hat{\mathcal{H}}),$ where $\mu_i, 1 \leq i \leq r$ are the roots of the polynomial  $$ (\mu-\theta-k)(1-\mu)^{r-1}+k.$$
			\item If $\beta$ is an eigenvalue of $\mathcal{Q}(\mathcal{H}),$ then $\mu_i \neq 1$ are the eigenvalues of $\mathcal{Q}(\hat{\mathcal{H}}),$ where $\mu_i, 1 \leq i \leq r$ are the roots of the polynomial     $$ (\mu-\beta-k)(\mu-1)^{r-1}-k.$$
		\end{itemize}
	\end{theorem}
	\begin{proof}
		The construction of the eigenvector is similar to the case of Theorem \ref{Corona_one_edge}.
	\end{proof}

	\begin{corollary}\footnote{This result for the case of graphs was observed by Dr. Sasmita Barik and presented in her lectures during the International Workshop on Special Matrices, Graphs and Applications-2025 organized by CARAMS, MAHE, Manipal.}
		Given a graph $G$ on $n$ vertices, let $\hat{G}(k), k \geq 1$ be obtained from $G$ by adding $k$ pendant edges at each vertex of $G$. If $\lambda$ is an eigenvalue $\hat{G}$, then $\frac{-k}{\lambda}$ is also an eigenvalue of $\hat{G}$.
	\end{corollary}

	\begin{theorem}
		Given two $r$-uniform hypergraphs $\mathcal{H}_1$ and $\mathcal{H}_2$ with order $n_1$ and $n_2$, respectively.
		Let $\mathcal{H} = \mathcal{H}_1 \odot \mathcal{H}_2$ be the corona of $\mathcal{H}_1$ with $\mathcal{H}_2$. 
		If $\theta$ is an eigenvalue of $\mathcal{L}(\mathcal{H}_1)$, then $\mu_i ~(\neq \binom{n_2-1}{r-2})$ are the eigenvalues of $\mathcal{L}(\mathcal{H})$, where $\mu_i, 1\leq i \leq r$ are roots of the polynomial $$\left(\mu-\theta-\binom{n_2}{r-1}\right)\left(\binom{n_2-1}{r-2}-\mu\right)^{r-1}+\binom{n_2}{r-1}\binom{n_2-1}{r-2}^{r-1}.$$
		Furthermore, if $\phi$ is an eigenvalue of $\mathcal{L}(\mathcal{H}_2),$ then $\phi+\binom{n_2-1}{r-2}$ is also an eigenvalue of $\mathcal{L}(\mathcal{H})$ with Hspan multiplicity at least $n_1$.
	\end{theorem}
	\begin{proof}
		Let $[(n_2+1)n_1]$ be the vertex set of $\mathcal{H} = \mathcal{H}_1 \odot \mathcal{H}_2$, the corona of $\mathcal{H}_1$ with $\mathcal{H}_2$.
		Without loss of generality, assume that $[n_1]$ is the vertex subset corresponding to $\mathcal{H}_1$ and $\{  n_1 + (j-1)n_2+ 1, n_1 + (j-1)n_2+ 2,\ldots, n_1+jn_2\}$ be the vertices corresponding to the $j^{th}$ ($1 \leq j \leq n_1$) copy of $\mathcal{H}_2$ in $\mathcal{H}$.
		Suppose that $(\theta, \mathbf{x})$ is an eigenpair of $\mathcal{L}(\mathcal{H}_1)$, and $\mu = \mu_i (\neq \binom{n_2-1}{r-2})$ is a root of the polynomial 
		$\left( \mu-\theta-\binom{n_2}{r-1}\right)\left(\binom{n_2-1}{r-2}-\mu \right)^{r-1} + \binom{n_2}{r-1}\binom{n_2-1}{r-2}^{r-1}.$
		Define $\mathbf{y} = \mathbf{y}^{(i)} := \left[ \mathbf{x}^\intercal, \frac{\binom{n_2-1}{r-2} x_1}{\binom{n_2-1}{r-2}-\mu_i}\mathbf{1}_{n_2}^\intercal, \ldots , \frac{\binom{n_2-1}{r-2}x_{n_1}}{\binom{n_2-1}{r-2} - \mu_i}\mathbf{1}_{n_2}^\intercal \right].$
		Then for any $j \in [n_1],$
		\begin{align*}
			\mu_i y_j^{r-1} - \left( \mathcal{L}(\mathcal{H} \mathbf{y}^{r-1}) \right)_j   &=  \mu_i y_j^{r-1}-d_{\mathcal{H}}(j) y_j^{r-1} + \sum\limits_{e \in \mathcal{E}_j(\mathcal{H})} \mathbf{y}^{ e \setminus \{j \} }\\
			&= \left(\mu_i-\theta-\binom{n_2}{r-1}\right)x_j^{r-1} + \binom{n_2}{r-1} \left( \frac{\binom{n_2-1}{r-2} x_j}{\binom{n_2-1}{r-2}-\mu_i}\right)^{r-1}=0
		\end{align*}
		Also, for any $k \in \{ j n_1+ 1, j n_1+2,\ldots, (j+1)n_1\},$ $j \in [n_1]$,
		\begin{align*}
			\mu_i y_k^{r-1} - \left( \mathcal{L}(\mathcal{H} \mathbf{y}^{r-1}) \right)_k   &=  \mu_i y_k^{r-1} - d_{\mathcal{H}}(k) y_k^{r-1} + \sum\limits_{e \in \mathcal{E}_k(\mathcal{H})} \mathbf{y}^{ e \setminus \{ k\} } \\
			&= \mu_i \frac{\binom{n_2-1}{r-2}}{\binom{n_2-1}{r-2}-\mu_i} x_j^{r-1} - \left(d_{\mathcal{H}_2}(k) + \binom{n_2-1}{r-2}\right) \left(\frac{\binom{n_2-1}{r-2}}{\binom{n_2-1}{r-2}-\mu_i}\right)^{r-1} x_j^{r-1}  \\
			&~~~~~~~~~~~~~+ d_{\mathcal{H}_2}(k)\left( \frac{\binom{n_2-1}{r-2}}{\binom{n_2-1}{r-2}-\mu_i}  \right)^{r-1} + \binom{n_2-1}{r-2} \left(\frac{\binom{n_2-1}{r-2}}{\binom{n_2-1}{r-2}-\mu_i}\right)^{r-2} x_j^{r-1} =0
		\end{align*}
		\noindent Note that, suppose $\phi$ is a non-main eigenvalue of $\mathcal{L}(\mathcal{H}_2)$ with non-main eigenvector $\mathbf{z}$, then $\phi+\binom{n_2-1}{r-2}$ is also an eigenvalue of $\mathcal{L}(\mathcal{H})$ with eigenvectors
		$\left[\begin{matrix}
			\mathbf{0} \\ \mathbf{z} \otimes \mathbf{e_j} 
		\end{matrix}\right], 1 \leq j \leq n_1.$
	\end{proof}

	\begin{theorem}
		Given an $r$-uniform hypergraph $\mathcal{H}_1$ on $n_1$ vertices and a $d$-regular, $r$-uniform hypergraph $\mathcal{H}_2$ on $n_2$ vertices, let $\mathcal{H}= \mathcal{H}_1 \odot \mathcal{H}_2. $
		\begin{itemize}
			\item If $\lambda$ is an eigenvalue of $\mathcal{A}(\mathcal{H}_1)$, then $\mu_i ~(\neq d)$ are the eigenvalues of $\mathcal{A}(\mathcal{H})$, 
			where $\mu_i, 1\leq i \leq r$ are the roots of the polynomial $(\mu -\lambda)(\mu-d)^{r-1} - \binom{n_2}{r-1}\binom{n_2-1}{r-2}^{r-1}$.
			Furthermore, if $\eta_i$ is a non-main eigenvalue of $\mathcal{A}(\mathcal{H}_2)$, then $\eta_i$ is also an eigenvalue of $\mathcal{A}(\mathcal{H})$ with Hspan multiplicity at least $n_1$.
			\item If $\beta$ is an eigenvalue of $\mathcal{Q}(\mathcal{H}_1)$, then $\mu_i ~(\neq 2d+\binom{n_2-1}{r-1})$ are the eigenvalues of $\mathcal{Q}(\mathcal{H})$, where $\mu_i, 1\leq i \leq r$ are the roots of the polynomial 
			$$\left(\mu-\beta-\binom{n_2}{r-1}\right)\left(\mu-2d-\binom{n_2-1}{r-1}\right)^{r-1}-\binom{n_2}{r-1}\binom{n_2-1}{r-2}^{r-1}.$$
			Furthermore, if $\alpha$ is a non-main eigenvalue of $\mathcal{Q}(\mathcal{H}_2)$, then $\alpha + \binom{n_2-1}{r-2}$ is also an eigenvalue of $\mathcal{Q}(\mathcal{H})$ with Hspan multiplicity at least $n_1$.
		\end{itemize}
	\end{theorem}
	\begin{proof}
		Let $[(n_2+1)n_1]$ be the vertex set of $\mathcal{H} = \mathcal{H}_1 \odot \mathcal{H}_2$, the corona of $\mathcal{H}_1$ with $\mathcal{H}_2$ ($d$-regular). Without loss of generality, assume that $[n_1]$ is the vertex subset corresponding to $\mathcal{H}_1$ and $\{  n_1 + (j-1)n_2+ 1, n_1 + (j-1)n_2+ 2,\ldots, n_1+jn_2\}$ be the vertices corresponding to the $j^{th}$ ($1 \leq j \leq n_1$) copy of $\mathcal{H}_2$ in $\mathcal{H}$. Suppose that $(\lambda,\mathbf{x})$ is an eigenpair of $\mathcal{A}(\mathcal{H}_1)$, and $\mu=\mu_i \neq d$ are the roots of the polynomial $(\mu-\lambda)(\mu-d)^{r-1}-\binom{n_2}{r-1}\binom{n_2-1}{r-2}^{r-1}$.
		Define $\mathbf{y}=\mathbf{y}^{(i)}:=\left[\mathbf{x}^\intercal, \frac{\binom{n_2-1}{r-2}x_1}{\mu_i-d} \mathbf{1}_{n_2}^\intercal, \frac{\binom{n_2-1}{r-2}x_2}{\mu_i-d} \mathbf{1}_{n_2}^\intercal, \ldots,  \frac{\binom{n_2-1}{r-2}x_{n_1}}{\mu_i-d} \mathbf{1}_{n_2}^\intercal \right]^{\intercal}$.
		We claim that $(\mu, \mathbf{y})$ 
		forms an eigenpair of $\mathcal{A}(\mathcal{H})$. 
		For any $j \in [n_1]$, we have 
		\begin{align*}
			\mu_i y_j^{r-1}- (\mathcal{A}(\mathcal{H})\mathbf{y}^{r-1})_j 
			&=  \mu_i x_j^{r-1}-\sum\limits_{e \in \mathcal{E}(\mathcal{H}_1)} \mathbf{x}^{e\setminus \{ j \}} - \sum\limits_{\substack{S \subseteq \{ n_1 + (j-1)n_2+ 1,\ldots, n_1+jn_2\}\\ |S|=r-1}} \mathbf{y}^{S} \\
			&=  \mu_i x_j^{r-1}-\lambda x_j^{r-1} - \frac{\binom{n_2}{r-1}\binom{n_2-1}{r-2}^{r-1} x_j^{r-1}}{(\mu_i-d)^{r-1}} =0.
		\end{align*}
		For any $k \in \{ j n_1+ 1, j n_1+2,\ldots, (j+1)n_1\},$ $j \in [n_1]$, we have 
		\begin{align*}
			\mu_i y_k^{r-1} - (\mathcal{A}(\mathcal{H})\mathbf{y}^{r-1})_k 
			&= \mu_i \frac{\binom{n_2-1}{r-2}^{r-1} x_j^{r-1}}{(\mu_i-d)^{r-1}} - \sum\limits_{e \in \mathcal{E}(\mathcal{H}_1)} \mathbf{y}^{e\setminus \{ j \}} - \sum\limits_{e \in \mathcal{E}(\mathcal{H}) \setminus \mathcal{E}(\mathcal{H}_1)} \mathbf{y}^{e\setminus \{ j \}} \\
			&= \mu \frac{\binom{n_2-1}{r-2}^{r-1} x_j^{r-1}}{(\mu-d)^{r-1}}- \frac{d\binom{n_2-1}{r-2}^{r-1} x_j^{r-1}}{(\mu-d)^{r-1}} - \frac{\binom{n_2-1}{r-2}^{r-1} x_j^{r-1}}{(\mu-d)^{r-2}} = 0.
		\end{align*}
		Furthermore, if there exist an eigenvector $\mathbf{z}$ of $\mathcal{A}(\mathcal{H}_2)$ associated with the eigenvalue $\eta$ satisfying Equation \ref{non-main_condition}, then we can see that $\eta$ is also an eigenvalue of $\mathcal{A}(\mathcal{H})$ with corresponding eigenvectors $\left[\begin{matrix}
			\mathbf{0} \\ \mathbf{z} \otimes \mathbf{e_j} 
		\end{matrix}\right], 1 \leq j \leq n_1$, where $e_j, 1 \leq j \leq n_1$ are the canonical vectors of length $n_1$.  
		
		Now, suppose that $(\beta, \mathbf{w})$ is an eigenpair of $\mathcal{Q}(\mathcal{H}_1)$, and $\mu = \mu_i~ ( \neq 2d+\binom{n_2-1}{r-2} )$ is a root of the polynomial \\ $\left( \mu-\beta-\binom{n_2}{r-1} \right)\left( \mu-2d-\binom{n_2-1}{r-2} \right)^{r-1}-\binom{n_2}{r-1} \binom{n_2-1}{r-2}$. Define $\mathbf{y}=\mathbf{y}^{(i)} := \left[ \mathbf{w}^\intercal, \frac{\binom{n_2-1}{r-2}w_1}{\mu_i-2d-\binom{n_2-1}{r-2}}\mathbf{1}_{n_2}^{\intercal}, \ldots , \frac{\binom{n_2-1}{r-2}w_{n_1}}{\mu_i-2d-\binom{n_2-1}{r-2} }\mathbf{1}_{n_2}^{\intercal} \right]^\intercal$. We claim that $(\mu_i, \mathbf{y}^{(i)})$ forms an eigenpair for $\mathcal{Q}(\mathcal{H})$.
		It is important to note that for any vertex $j \in [n_1]$,  $d_{\mathcal{H}}(j) = d_{\mathcal{H}_1}(j) + \binom{n_2}{r-1},$ and for any vertex $k \in [n_1(n_2+1)] \setminus [n_1],$ $d_{\mathcal{H}}(k)= d_{\mathcal{H}_2}(k) + \binom{n_2-1}{r-2}.$ 
		Proof of the claim is similar to the case of adjacency hypermatrix.
		
		\noindent Suppose $\eta$ is a non-main eigenvalue of $\mathcal{Q}(\mathcal{H}_2)$ with non-main eigenvector $\mathbf{z}$, then $\eta+\binom{n_2-1}{r-2}$ is also an eigenvalue of $\mathcal{Q}(\mathcal{H})$ with eigenvectors
		$\left[\begin{matrix}
			\mathbf{0} \\ \mathbf{z} \otimes \mathbf{e_j} 
		\end{matrix}\right], 1 \leq j \leq n_1.$
	\end{proof}

	\section{Reciprocal Eigenvalue Property of Hypertrees}
	
	In this section, we generalize the study of reciprocal eigenvalue property for uniform hypergraphs, and we obtain the negative result in the case of power hypertrees.
	
	Since zero is an eigenvalue of every $r$-uniform hypergraph for $r \geq3$, we say that  a hypergraph $\mathcal{H}$ satisfies the \emph{reciprocal eigenvalue property} ((R) property) if for each $0 \neq \mu \in \sigma(\mathcal{H})$, there exist $\mu' \in \sigma(\mathcal{H})$ such that $\mu \mu' =1.$
	
	Similarly, we say that a hypergraph $\mathcal{H}$ satisfies the \emph{strong reciprocal eigenvalue property} ((SR) property) if for each $0 \neq \mu \in \sigma(\mathcal{H})$ with multiplicity $m(\mu)$, there exist $\mu' \in \sigma(\mathcal{H})$ with multiplicity $m(\mu')= m(\mu)$  such that $\mu \mu' = 1.$

	\begin{remark}
		The power hypergraph of a corona graph need not satisfy the reciprocal eigenvalue property.
	\end{remark}
	
	\begin{example}
		We have 
		\begin{equation*}
			\Phi_{P_3^{(3)}}(\lambda)= \lambda^{259}(\lambda^3-1)^{27}(\lambda^3-2)^{18}(\lambda^6-3\lambda^3+1)^9, 
		\end{equation*}
		\begin{equation*}
			\Phi_{P_3^{(4)}}(\lambda)= \lambda^{95774}(\lambda^4-1)^{11440}(\lambda^4-2)^{5632}(\lambda^8-3\lambda^4+1)^{4096}. 
		\end{equation*}
	\end{example}

	A signed graph is a graph together with a sign assigned to every edges of the graph. Formally, a signed graph is a triple $G_{\pi} = (V,E,\pi)$, with $V$ and $E$ being the vertex and edge set of the graph and $\pi$ is a function from $E$ to $\{ +1,-1 \}.$

	\begin{theorem}\cite{chen2023all}\label{eigenvalues_of_power}
		The $r$-power hypergraph of a graph $G$, denoted by $G^r$  has an eigenvalue $\lambda$ if and only if 
		\begin{itemize}
			\item some signed induced subgraph of $G$ has an eigenvalue $\beta$ such that $\beta^2 = \lambda^r$, if $r =3$;
			\item some signed subgraph of $G$ has an eigenvalue $\beta$ such that $\beta^2= \lambda^r$, if $r \geq 4$.
		\end{itemize}
	\end{theorem}
	
	\noindent Given a hypergraph $\mathcal{H}$ of order $n$ and size $m$, the matching polynomial of $\mathcal{H}$ in $\lambda$, denoted by $\varphi_{\mathcal{H}}(\lambda)$ is defined as 
	$$\varphi_{\mathcal{H}}(\lambda) = \sum\limits_{k \geq 0} (-1)^k m_k(\mathcal{H}) \lambda^{n-kr},$$ where $m_k(\mathcal{H})$ is the number of $k$-matchings in $\mathcal{H}$.
	
	\begin{theorem}\cite{li2024relationship}\label{relation_between_matching_characteristic}
		For a $r$-uniform hypertree $\mathcal{T}$ on $m$ hyperedges with $r\geq 2$,
		$$\Phi_{\mathcal{T}}(\lambda)= \prod\limits_{\substack{{H} \text{ is a labeled} \\ \text{ connected subhypergraph of } \mathcal{T}}} [\varphi_{H}(\lambda)]^{a_{{H}}},$$
		where  
		$a_H=(r-1)^{(r-1)(m-e(H)-|\partial(H)|)} r^{(r-2)(e(H))} ((r-1)^{r-1}-r^{r-2})^{|\partial(H)|},$ where $\partial(H) $ is the set of edges of $\mathcal{T}$ that are incident with the vertices of both $\mathcal{V}(\mathcal{H})$ and $\mathcal{V}(\mathcal{T})\setminus \mathcal{V}(H)$.
	\end{theorem}

	

	
	\begin{proposition}
		An $r$-power hypertree with $r\geq 3$ does {\bf not} satisfy weak reciprocal eigenvalue property.
	\end{proposition}
	\begin{proof}
		Uisng the notations from Theorem \ref{relation_between_matching_characteristic},  for any connected subhypergraph $H$ of a $r$-uniform hypertree with $r \geq 3$, note that 
		$a_H=0$ if and only if $r=2$.
		As we already have the relationship between the matching polynomial and the characteristic polynomial of a $r$-uniform hypertree, it can be observed that for $r \geq 3,$ $(\lambda^{2r-1}-2\lambda^{r-1})^{t}$ is always a factor in the characteristic polynomial of the hypertree with $t \geq 1$. Hence the proof follows by combining Theorem \ref{eigenvalues_of_power}, \ref{relation_between_matching_characteristic} and Rational Root Theorem.
	\end{proof}
	
	In the following result we show that the non-zero part of the spectrum of any $r$-uniform hypertree does not satisfy reciprocal eigenvalue property.
	
	\begin{proposition}
		Given an $r$-uniform hypertree $\mathcal{T}$ with $r \geq 3$ and $m \geq 2$, if $\eta$ is the nullity of $\mathcal{T}$, then $\Phi_{\mathcal{T}}(\lambda) / \lambda^{\eta}$ is {\bf not} palindromic.
	\end{proposition}
	\begin{proof}
		Suppose that $\mathcal{T}$ is an $r$-uniform hypertree.
		Since $r \geq 3,$ we have that $(\lambda^{2r-1}-2\lambda^{r-1})$ is always a factor in the characteristic polynomial of the hypertree. 
		If $\eta$ is the nullity of the hypertree $\mathcal{T}$ and $a_{n-\eta}$ is the coefficient of $\lambda^{\eta}$ in the characteristic polynomial, then $a_{n-\eta} \geq 2$. Hence we have $\Phi_{\mathcal{T}}(\lambda) \neq \Phi_{\mathcal{T}}(\frac{1}{\lambda})$.
	\end{proof}

	\subsection{A Generalization of Reciprocal Eigenvalue Property for Hypergraphs}

	In view of the results obtained above, on relaxing the condition, we propose a generalization of reciprocal eigenvalue property for hypergraphs as follows:
	\begin{itemize}
		\item Find the maximal subset (multi-set) of the spectrum of  hypergraph satisfying the (R) property and (SR) property. 
	\end{itemize}
	
	\begin{remark}\label{char_match_remark}
		Spectrum of an $r$-uniform hypertree with $r \geq 3$ satisfies (SR) property if and only if the matching polynomial of every sub-hypertree of the hypertree is palindromic.
	\end{remark}
	
	Following are the simple observations using the above remark. 
	
	\begin{proposition}
		For $r\geq 3$ and $m \geq 3$, there does not exist any $r$-uniform hypertree with $m$ hyperedges such that $spec(\mathcal{T}) = spec(\mathcal{S}_m^{(r)})$ and $\mathcal{T} \not\cong \mathcal{S}_m^{(r)}.$
	\end{proposition}
	\begin{proof}
		The direct proof is by looking at the nullity of the hyperstar \cite{zheng2025uniform}, which characterizes it completely.
	\end{proof}

	\begin{proposition}
		For an $r$-uniform hypertree $\mathcal{T}$, if the maximal subset of the spectrum that is closed under inverse is $\{ {\zeta^i}^{(m(1))} | 0 \leq i \leq r-1, \zeta^r=1 \},$  then $\mathcal{T}$ is isomorphic to a hyperstar. 
		That is, for an $r$-uniform hypertree, if the only self-reciprocal factor of $\Phi_{\mathcal{T}}(\lambda)$ is $(\lambda^r-1)^{m(1)},$ then $\mathcal{T}$ is isomorphic  to a hyperstar. 
	\end{proposition}
	\begin{proof}
		Let $\mathcal{T}'$ be a $r$-uniform hypertree, whose only self-reciprocal factor is $(\lambda^r-1)^{m(1)}$. Suppose that $\mathcal{T}'$ is not isomorphic to a hyperstar, then $\mathcal{T}'$ induces a connected sub-hypergraph isomorphic to $\mathcal{P}_3^{(r)}$. By using Theorem \ref{relation_between_matching_characteristic}, $(\lambda^{2r}-3\lambda^r+1)$ is a factor in $\Phi_{\mathcal{T}'}(\lambda),$ a contradiction.  
	\end{proof}
	
	\begin{proposition}
		For an $r$-uniform hypertree $\mathcal{T}$ with at least two hyperedges, if the (multi) subset $B=spec(\mathcal{T}) \setminus (\{0 ^{(\eta)}\} \cup \{ \zeta^i \sqrt[r]{2}^{(m(\sqrt[r]{2}))} | 0 \leq i \leq r-1, \zeta^r=1\})$ of the spectrum of $\mathcal{T}$ is closed under inverse, then $\mathcal{T} \cong \mathcal{P}_2^{(r)}$ or $\mathcal{T} \cong \mathcal{P}_3^{(r)}$.  
	\end{proposition}
	\begin{proof}
		By direct computation, it can be seen that if $\mathrm{T}$ is isomorphic to $\mathcal{P}_2^{(r)}$ or $\mathcal{P}_3^{(r)}$, then the multi-set $B$ is closed under inverse.
		When $m=2$, the hypertree is isomorphic to $\mathcal{P}_2^{(r)} \cong \mathcal{S}_2^{(r)}.$
		When $m=3$, the hypertree is either isomorphic to $\mathcal{P}_3^{(r)}$ or $\mathcal{S}_3^{(r)}.$ Suppose $\mathcal{T} \cong \mathcal{S}_3^{(r)}; $ $B$ is not closed under inverse, as $3^{\frac{1}{r}} \in B$ and $3^{-\frac{1}{r}}  \not\in B$.
		Let us consider the case when $m \geq 4$. 
		Suppose $\mathcal{T} \cong \mathcal{P}_m^{(r)}$; then $3^{\frac{1}{r}} \in B$ and by using Theorem \ref{eigenvalues_of_power} and Rational Root Theorem, $3^{-\frac{1}{r}}  \not\in B$. 
		The case is similar if $\mathcal{T} \cong \mathcal{S}_m^{(r)}.$
		Suppose $\mathcal{T} \not\cong \mathcal{P}_m^{(r)}$ and $\mathcal{T} \not\cong \mathcal{S}_m^{(r)}$; then $\mathcal{T}$ must have one of the following two as an induced sub-hypergraph: 
		
		\noindent Case 1. A hyperedge attached to a pendant vertex of a non-pendent hyperedge of $\mathcal{P}_3^{(r)}$, the depiction of which has shown in Figure \ref{Case1}.
		
		\begin{figure}[H]
			\centering
			\includegraphics[width=0.2\linewidth]{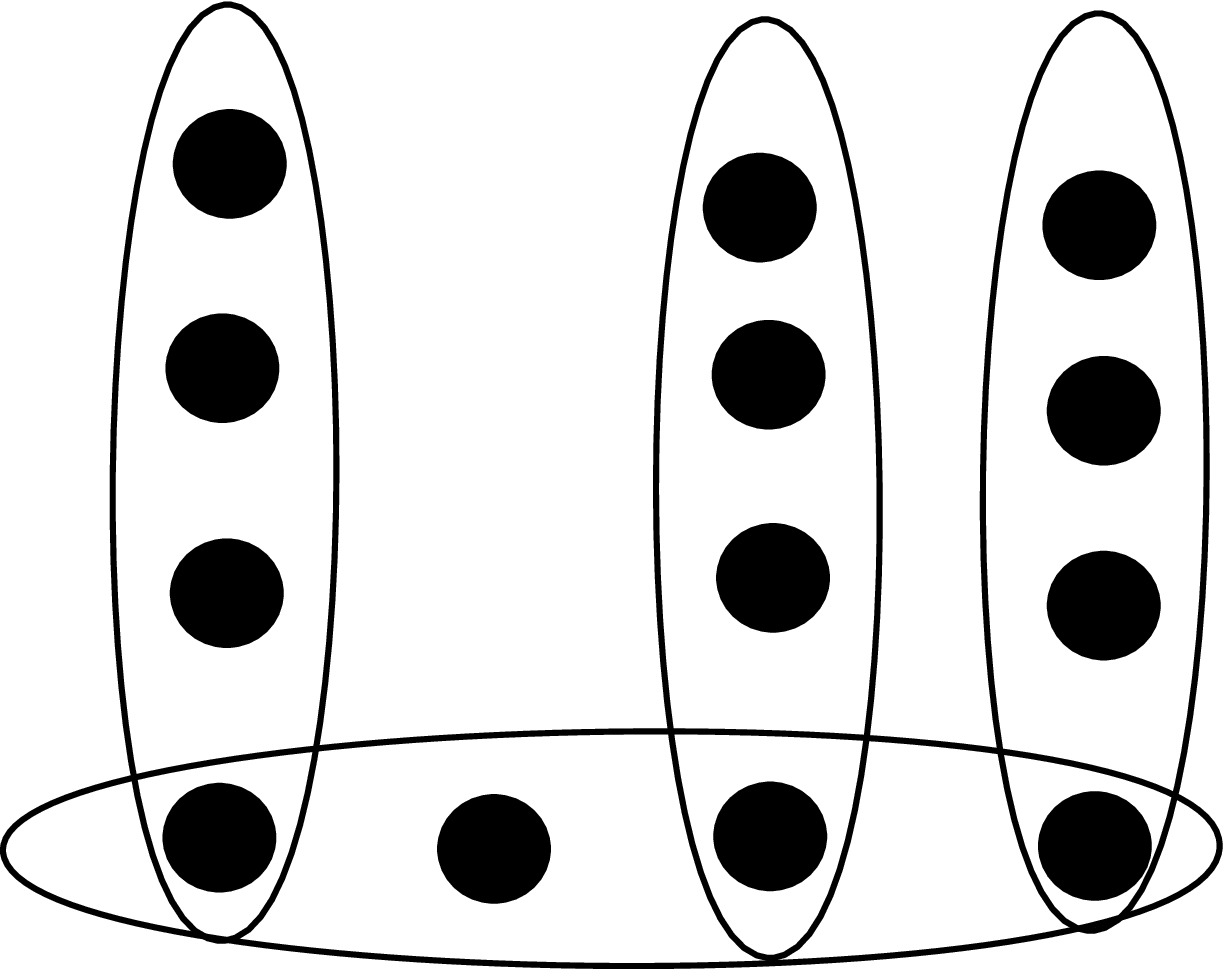}
			\caption{Sub-hypertree $\mathcal{T}_1$}
			\label{Case1}
		\end{figure}
		
		\noindent Case 2. A hyperedge attached to a non-pendent vertex of $\mathcal{P}_3^{(r)}$, which is depicted in Figure \ref{Case2}.
		
		\begin{figure}[H]
			\centering
			\includegraphics[width=0.2\linewidth]{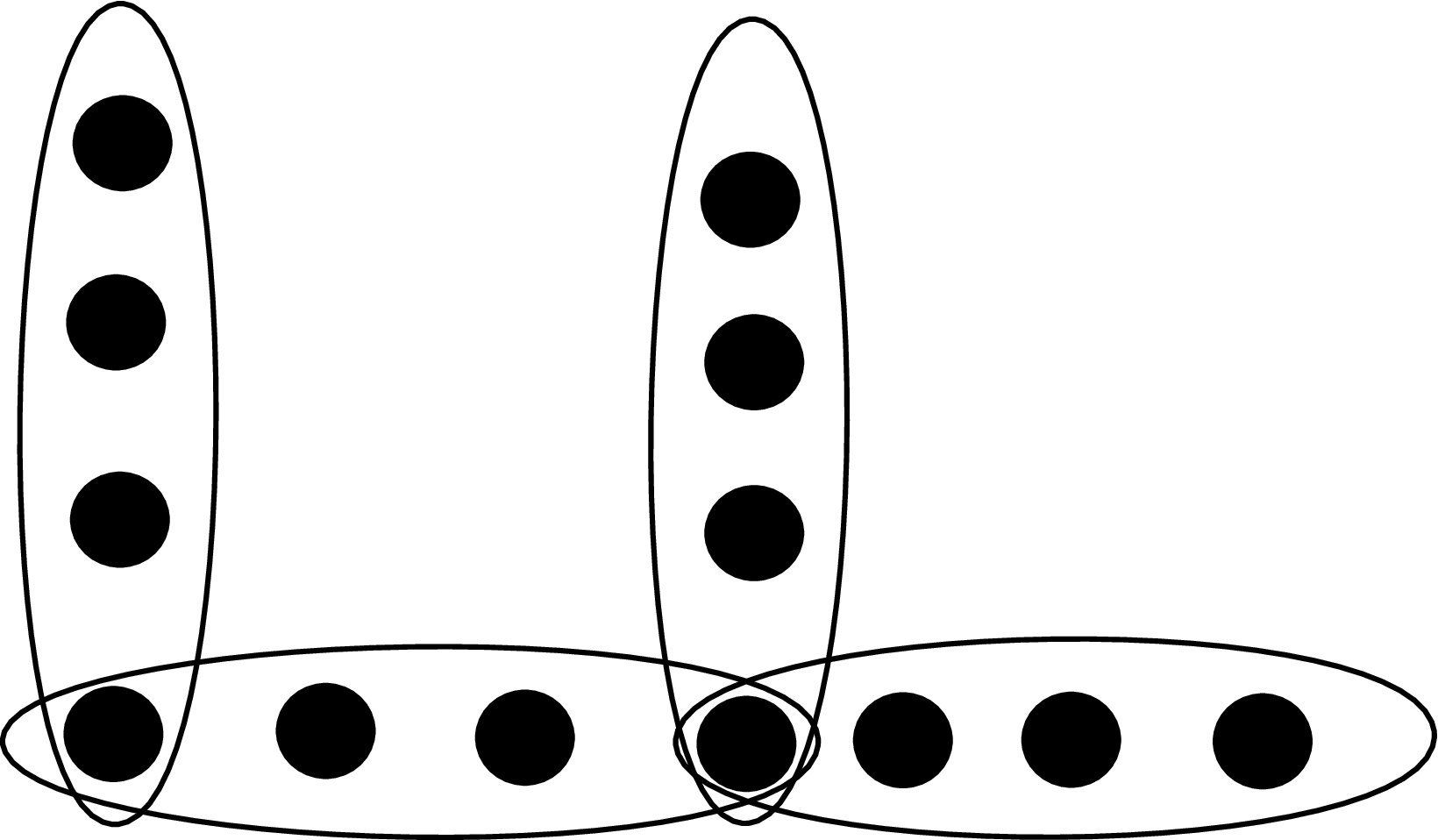}
			\caption{Sub-hypertree $\mathcal{T}_2$}
			\label{Case2}
		\end{figure}
		
		Suppose that $\mathcal{T}$ contains $\mathcal{T}_1$ as a sub-hypergraph, then $\Phi_{\mathcal{T}}(\lambda)$ contains $(\lambda^{2r}-4\lambda^r+3)$ as a factor, and hence $3^{\frac{1}{r}}$ as an eigenvalue of $\mathcal{T}.$ Even otherwise, if $\mathcal{T}$ contains $\mathcal{T}_2$ as an induced sub-hypergraph, then $3^{\frac{1}{r}}$ is an eigenvalue of $\mathcal{T}$. Now by using the Rational Root Theorem, the proof concludes.
		
	\end{proof}

	\section*{Acknowledgment}
	
	Authors acknowledge Manipal Academy of Higher Education for overall support.

	\footnotesize
	\bibliographystyle{unsrt}
	\bibliography{bib_file}


		
		


\end{document}